\documentclass{article}
\usepackage[twoside,
paperwidth=210mm,
paperheight=297mm,
textheight=622pt,
textwidth=468pt,
centering,
headheight=50pt,
headsep=12pt,
footskip=18pt,
footnotesep=24pt plus 2pt minus 12pt,
columnsep=2pc]{geometry}

\usepackage[auth-sc]{authblk}

\usepackage{mathtools}

\allowdisplaybreaks

\usepackage{amssymb}
\usepackage[mathscr]{eucal}

\usepackage{mismath}

\usepackage{hyperref}

\usepackage{bm} 

\usepackage{graphicx}
\graphicspath{{figs/}}

\usepackage{pgf}
\usepackage{pgfplots}
\pgfplotsset{compat=1.14}

\usepackage{tikz}
\usetikzlibrary{math}
\usetikzlibrary{calc}

\usepackage{subcaption}

\usepackage{multirow}

\usepackage{makecell}
\setcellgapes{2pt}

\usepackage{booktabs}

\usepackage{siunitx}
\usepackage{algorithm}
\usepackage{algorithmic}

\usepackage[capitalise]{cleveref}

\DeclareMathOperator{\Dist}{dist}
\DeclareMathOperator{\Span}{span}

\DeclareMathOperator{\Diag}{diag}

\DeclarePairedDelimiter{\RoundBrackets}{(}{)}
\DeclarePairedDelimiter{\CurlyBrackets}{\{}{\}}
\DeclarePairedDelimiter{\SquareBrackets}{[}{]}

\usepackage{stmaryrd}

\DeclareSymbolFont{sfletters}{OML}{cmbrm}{m}{it}

\DeclareMathSymbol{\salpha}{\mathord}{sfletters}{"0B}
\DeclareMathSymbol{\sbeta}{\mathord}{sfletters}{"0C}
\DeclareMathSymbol{\sgamma}{\mathord}{sfletters}{"0D}
\DeclareMathSymbol{\sdelta}{\mathord}{sfletters}{"0E}
\DeclareMathSymbol{\sepsilon}{\mathord}{sfletters}{"0F}
\DeclareMathSymbol{\szeta}{\mathord}{sfletters}{"10}
\DeclareMathSymbol{\seta}{\mathord}{sfletters}{"11}
\DeclareMathSymbol{\stheta}{\mathord}{sfletters}{"12}
\DeclareMathSymbol{\siota}{\mathord}{sfletters}{"13}
\DeclareMathSymbol{\skappa}{\mathord}{sfletters}{"14}
\DeclareMathSymbol{\slambda}{\mathord}{sfletters}{"15}
\DeclareMathSymbol{\smu}{\mathord}{sfletters}{"16}
\DeclareMathSymbol{\snu}{\mathord}{sfletters}{"17}
\DeclareMathSymbol{\sxi}{\mathord}{sfletters}{"18}
\DeclareMathSymbol{\spi}{\mathord}{sfletters}{"19}
\DeclareMathSymbol{\srho}{\mathord}{sfletters}{"1A}
\DeclareMathSymbol{\ssigma}{\mathord}{sfletters}{"1B}
\DeclareMathSymbol{\stau}{\mathord}{sfletters}{"1C}
\DeclareMathSymbol{\supsilon}{\mathord}{sfletters}{"1D}
\DeclareMathSymbol{\sphi}{\mathord}{sfletters}{"1E}
\DeclareMathSymbol{\schi}{\mathord}{sfletters}{"1F}
\DeclareMathSymbol{\spsi}{\mathord}{sfletters}{"20}
\DeclareMathSymbol{\somega}{\mathord}{sfletters}{"21}
\DeclareMathSymbol{\svarepsilon}{\mathord}{sfletters}{"22}
\DeclareMathSymbol{\svartheta}{\mathord}{sfletters}{"23}
\DeclareMathSymbol{\svarpi}{\mathord}{sfletters}{"24}
\DeclareMathSymbol{\svarrho}{\mathord}{sfletters}{"25}
\DeclareMathSymbol{\svarsigma}{\mathord}{sfletters}{"26}
\DeclareMathSymbol{\svarphi}{\mathord}{sfletters}{"27}
\DeclareMathSymbol{\sDelta}{\mathord}{sfletters}{"01}
\DeclareMathSymbol{\sTheta}{\mathord}{sfletters}{"02}

\usepackage{amsthm}
\newtheorem{theorem}{Theorem}[section]
\newtheorem{lemma}[theorem]{Lemma}

\theoremstyle{definition}
\newtheorem{definition}{Definition}

\crefname{assumption}{assumption}{assumptions}
\Crefname{assumption}{Assumption}{Assumptions}

\crefname{problem}{problem}{problems}
\Crefname{problem}{Problem}{Problems}

\theoremstyle{remark}

\usepackage{lineno}

\usepackage{easyReview}

\title{Graph-Based Meshfree Multi-scale Coarse Space Approximation for Two-Level Schwarz Methods}
\author[1]{Yucheng~Liu}
\author[2]{Tak~Shing~Au~Yeung}
\author[1]{Eric~T.~Chung}
\author[2]{Simon~See}
\affil[1]{Department of Mathematics, The Chinese University of Hong Kong, Shatin, Hong~Kong~SAR, China.}
\affil[2]{NVIDIAAI Technology Center NVAITC, NVIDIA, USA}

\date{}
 
\begin{document}
\maketitle

\begin{abstract}
Efficient simulation of Darcy flow in highly heterogeneous porous media requires iterative solvers that remain robust under large permeability contrasts and mixed boundary conditions. 
Spectral coarse spaces in two-level overlapping Schwarz methods provide such robustness, but their practical use is often limited by an expensive setup phase dominated by many local generalized eigenvalue solves. We propose a purely algebraic,  coarse-space approximation that avoids these repeated local eigensolves by using a graph neural network operating on the system-matrix graph. 
On the analysis side, we introduce a coefficient-weighted subspace-distance measure to quantify the discrepancy between the approximated and target local multiscale coarse spaces, and we derive a condition-number bound for the resulting preconditioned operator in terms of this distance. 
This bound yields a principled supervised-training objective and links learning error to solver performance. 
Numerical experiments on 2D and 3D high-contrast Darcy systems with varying mixed boundary conditions demonstrate that the proposed approach substantially reduces setup cost and improves end-to-end time-to-solution, while preserving robust convergence across the tested contrasts and boundary configurations.

\textbf{Keywords}: multiscale method, coarse space, two-level schwarz methods

\textbf{MSC codes}: 65N08, 65N15, 65N55
\end{abstract}

\section{Introduction}

Accurate simulation of flow in porous media underpins reservoir forecasting, groundwater management, and geothermal exploration \cite{helmig1997multiphase,zheng2002applied,abou2013petroleum}.
A central difficulty is the multiscale variability of permeability, spanning pore-scale features to field-scale geologic structures.
Resolving this heterogeneity is crucial: overly homogenized models can yield unreliable predictions of fluxes and plume migration \cite{gelhar1986stochastic,rubin2003applied}.
For Darcy-flow models, additional complications arise from high-contrast permeability fields and mixed boundary conditions that vary with operating regimes.
After discretization, these features translate into strongly ill-conditioned linear systems with challenging spectra.
Because direct solvers are often limited by memory and computational complexity at scale, Krylov subspace methods are the workhorse; their practical performance, however, depends not only on iteration counts but also on the setup cost of the preconditioner and its robustness to large coefficient contrasts and changing boundary constraints.

Robust preconditioning for such problems has advanced considerably over the past decades.
Classical incomplete factorizations, e.g., ILU \cite{meijerink1977iterative,saad2003iterative}, are appealing for their simplicity and modest memory footprint, but they may deteriorate for highly heterogeneous coefficients and large-scale systems.
To reduce mesh-dependent convergence, algebraic multigrid (AMG) \cite{ruge1987algebraic} is widely used and often achieves near-optimal complexity for elliptic problems; nevertheless, standard coarsening and interpolation heuristics can lose robustness in high-contrast, strongly anisotropic, or channelized media \cite{wan1999energy}.
Domain decomposition offers a complementary route to scalability.
In particular, two-level overlapping Schwarz methods \cite{toselli2004domain,dolean2015introduction} show that robustness to coefficient variation hinges on the choice of coarse space.
In high-contrast settings where standard coarse spaces are inadequate, spectral multiscale techniques—such as the Generalized Multiscale Finite Element Method (GMsFEM) \cite{efendiev2011multiscale,efendiev2013generalized,chung2018constraint,chung2023multiscale,ye2024robust} and the GenEO framework \cite{spillane2014abstract,dolean2015introduction}—provide a principled remedy.
These methods enrich the coarse space by solving local generalized eigenproblems that identify the dominant low-energy modes.
To reduce dependence on mesh geometry and enable a unified treatment in 2D and 3D, these spectral ideas can be cast in a purely algebraic two-level overlapping Schwarz framework that uses only the graph of the system matrix \cite{heinlein2018frosch,ADDThm,ADD,al2025robust}.
Nevertheless, even in algebraic form, the setup phase can dominate the total cost because it requires solving many local eigenproblems, creating a bottleneck for end-to-end simulations.

Specifically, we avoid repeated local generalized eigensolves at setup by using a graph neural network (GNN) to approximate the local spectral coarse spaces from the system-matrix graph and associated features.
We introduce a subspace-distance measure between the learned and target local coarse spaces and show that the condition number of the resulting preconditioned operator can be bounded in terms of this distance along with standard Schwarz constants.
Guided by this bound, we train the GNN in a supervised fashion to minimize the subspace distance to the target spectral coarse space.
The approach enables an adaptive local coarse dimension across subdomains, allocating more basis vectors to regions associated with low-energy modes induced by high contrast or channelization.
The resulting two-level preconditioner is symmetric positive definite (SPD) by construction, and is therefore compatible with the Conjugate Gradient method.
By incorporating boundary-condition descriptors into features of the input graph, the trained model generalizes across a range of mixed boundary conditions in our tests without problem-specific tuning.
Numerical experiments on 2D and 3D Darcy problems show iteration counts comparable to exact spectral coarse spaces while substantially reducing setup time, leading to improved end-to-end time-to-solution.

The paper is organized as follows: \Cref{pre} introduces the purely algebraic two-level overlapping Schwarz preconditioner; \Cref{nn} develops the learning-based construction and analysis; \Cref{nume} reports numerical results; and \Cref{con} concludes with future directions.

\section{Preliminaries}\label{pre}

\subsection{Problem setting}
Consider the Darcy flow problem characterized by highly heterogeneous permeability.
Let \(\Omega \subset \mathbb{R}^d\), \(d \in \{2,3\}\), be a bounded Lipschitz domain with boundary \(\partial\Omega = \Gamma_D \cup \Gamma_N\), where \(\Gamma_D\) and \(\Gamma_N\) denote the Dirichlet and Neumann boundaries, respectively.
\begin{equation}\label{eq:original_equation}
\left\{
\begin{aligned}
\kappa^{-1}\bm{u}+\nabla p &= \bm{0} \quad &&\text{in } \Omega,\\
\nabla \cdot \bm{u} &= f \quad &&\text{in } \Omega,\\
p &= p_D \quad &&\text{on } \Gamma_D,\\
\bm{u}\cdot \bm{n} &= g \quad &&\text{on } \Gamma_N.
\end{aligned}
\right.
\end{equation}
Let \(N_\Omega\) denote the dimension of the discrete space associated with a finite volume discretization.
In this work, We focus on the large systems of linear algebraic equations
\[
A u = f, \qquad A \in \mathbb{R}^{N_\Omega \times N_\Omega}, \quad f \in \mathbb{R}^{N_\Omega},
\]
arising from this method.

\subsection{Domain decomposition and subdomain coverings}
Let $\mathscr{G}(A)$ denote the undirected adjacency graph associated with a sparse, symmetric, positive semidefinite matrix $A$, and index its vertex set $V$ by the integers $1$ through $N_\Omega$.
We partition $\mathscr{G}(A)$ into $k \ll N_\Omega$ nonoverlapping subgraphs using a graph partitioning method (e.g., METIS). 
Equivalently, this induces a partition of the vertex set into disjoint subsets $V_{I,i}$, $i \in [1,k]$.
For $\delta \in \mathbb{N}$, define $V_{\Gamma,i}^{\delta}$ to be the set of vertices in $V \setminus V_{I,i}$ whose graph distance from $V_{I,i}$ is within $\delta$.
We obtain, via the above graph partition, the corresponding decomposition of the physical domain into the interior subdomains $\Omega_{I,i}$ and the interface regions $\Omega^{\delta}_{\Gamma,i}$, where $i$ is from $1$ to $k$ \cref{fig:GP}. 
Accordingly, the overlapping subdomain and corresponding local vertex set are
\[
\Omega^{\delta}_{i} :=\Omega_{I,i} \cup \Omega^{\delta}_{\Gamma,i}, \quad V_i^{\delta} :=  V_{I,i} \cup V_{\Gamma,i}^{\delta}, \quad N_{i} = \abs{V_i^{\delta}}.
\]

\begin{figure}[!ht]
  \centering
  \includegraphics[width=0.8\textwidth]{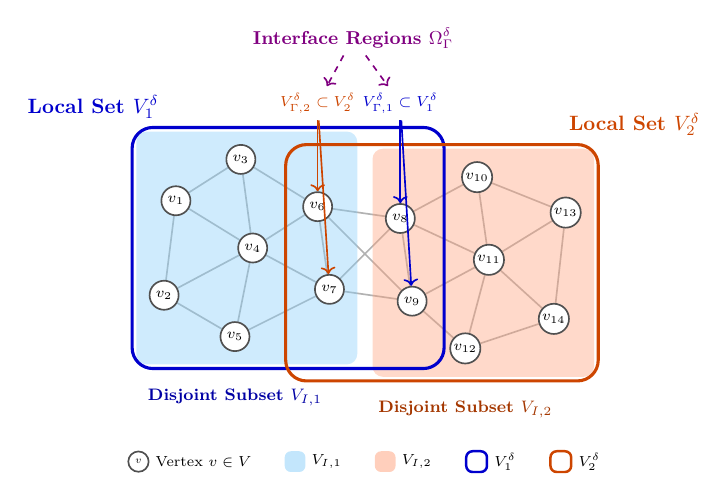}
  \caption{An illustration of the overlapping graph partition: the vertex set $V$ is first partitioned into disjoint subsets $V_{I,1}$ and $V_{I,2}$ (shaded regions). The local vertex sets $V_1^\delta$ and $V_2^\delta$ (outlined regions) are then formed by including the halo vertices $V_{\Gamma,i}^\delta$ within graph distance $\delta$.}
  \label{fig:GP}
\end{figure}

For each vertex $v\in V$, define
\[
 d_v := \# \{V_i^{\delta}: 1 \leq i \leq k , v \in V_i^{\delta}\},
\]
so that $v$ belongs to $d_v$ of the local vertex sets.
For each overlapping subdomain $\Omega^{\delta}_{i}$, denote its complement in $\Omega$ by $\Omega^{\delta}_{c,i}$.
Based on the graph partitioning, let $R_i$ denote the restriction operator from the global domain to the $i$-th overlapping subdomain, that is, $R_i : \Omega \to \Omega_i^{\delta}$.
Using the vertex degrees ${d_v}$, we define an algebraic partition of unity via positive diagonal matrices $D_i$ such that
\[
\sum_{i=1}^k  R_{i}^\top D_{i}  R_{i} = I_n.
\]

\subsection{Local overlapping solvers}
In each overlapping subdomain $\Omega_i^{\delta}$, define the local operator $A_i := R_i A R_i^{\top}$. 
Given a global residual $r$, the local overlapping solvers is obtained by solving the subproblem $A_i u_i = R_i r$ and prolongating back via $R_i^{\top}$.
The first-level of additive Schwarz preconditioner aggregates these local corrections
\begin{equation}\label{ams1}
  M_1^{-1} = \sum_{i=1}^k R_{i}^\top D_iA_i^{-1} R_{i}.
\end{equation}

\subsection{Coarse space correction}
Based on the overlapping partitioning of $A$, let $P_i$ be the permutation that orders the unknowns as $(I, \Gamma, c)$ with respect to subdomain $i$.
we have local block splitting matrix for each $i \in [1, k]$,
\[
P_i AP_i^T = \begin{pmatrix} 
A_{I I,i} & A_{I \Gamma,i} & 0\\ 
A_{\Gamma I,i} & A_{\Gamma \Gamma,i} & A_{\Gamma c,i}\\
0 & A_{c\Gamma,i} & A_{cc,i} 
\end{pmatrix}.
\]
Define $s_i \in R^{\abs{V_{\Gamma,i}^{\delta}}}$
\[
(s_i)_j = \sum_l \abs{(A_{\Gamma c,i})_{jl}}, \quad S_i := \Diag(s_i),
\]
and set 
\[
\tilde{A}_{\Gamma\Gamma,i} = A_{\Gamma\Gamma,i} - S_i.
\]
The corrected local block is then
\[
\tilde{A}_{i} = \begin{pmatrix} 
A_{II,i} & A_{I \Gamma,i} \\ 
A_{\Gamma I,i} & \tilde{A}_{\Gamma \Gamma,i}
\end{pmatrix}.
\]

Two bilinear forms $a_i(\cdot, \cdot)$ and $\tilde{a}_i(\cdot, \cdot)$ corresponding to the $i$-th subdomain are defined to construct local generalized multiscale coarse space, where
\[
a_i(u, v) = v^{\top} \Pi_i D_i A_{i} D_i \Pi_i u, \quad 
\tilde{a}_i(u, v) = v^{\top} \tilde{A}_{i} u, \quad \forall u,v \in \mathbb{R}^{N_i}
\]
and \( \Pi_i \) is the orthogonal projection on image space of $\tilde{A}_{i}$, i.e., $\Pi_i:\mathbb{R}^{N_i}\to \operatorname{Im}(\tilde A_i)$. Now, define the following local generalized spectral problem:
\begin{equation}\label{eq:a_i_func}
a_i(u, v) = \lambda \tilde{a}_i(u, v), \quad \forall v \in \mathbb{R}^{n_i}.
\end{equation}
The local generalized multiscale coarse space is 
\[
X_i = \Span \left\{ u \mid a_i(u, v) = \lambda \tilde{a}_i(u, v), \, \forall v \in \mathbb{R}^{n_i}, \, |\lambda| \geq \tau \right\}.
\]
In this paper, for each subdomain we select the $n^c$ eigenvectors associated with the largest $n^c$ generalized eigenvalues to form the local generalized multiscale coarse space; we denote the $n^c$-th largest eigenvalue on subdomain $i$ by $\tau_i$.

Finally, we define the local coarse space by taking the Euclidean orthogonal complement within $\ker(\tilde A_i)$:
\[
\big((\ker(D_i A_{i} D_i) \cap \ker(\tilde{A}_{i}))^{\perp} \cap \ker(\tilde{A}_{i})\big)\oplus X_i,
\]
The basis of the $i$-th local coarse space is assembled as the column matrix $Z_i$.
The global restriction for the coarse correction is assembled as
\[
R_0^{\top} = \begin{bmatrix}
R_1^{\top} D_1 Z_1 & \cdots & R_k^{\top} D_k Z_k
\end{bmatrix}.
\]

Given a global residual $r$, the coarse space correction is obtained by solving the subproblem $R_0 A R_0^{\top} u_0 = R_i r$ and prolongating back via $R_0^{\top}$. 
The second-level of additive Schwarz preconditioner is 
\begin{equation}\label{ams2}
  M_2^{-1} = R_{0}^\top (R_{0}AR_{0}^\top)^{-1} R_{0}.
\end{equation}

\subsection{Albebraic two-level schwarz preconditioner}
The algebraic two-level Schwarz preconditioner is defined by the additive combination of the local overlapping level and the coarse level:
\begin{equation}\label{ams}
    M_{AMS}^{-1} = M_1^{-1} + M_2^{-1}.
\end{equation}

\section{Methodology}\label{nn}
In the construction of an algebraic two-level Schwarz preconditioner, a collection of local generalized spectral problems \cref{eq:a_i_func} must be solved on overlapping subdomains to generate the multiscale subspaces used to build the coarse space.
As reviewed in the preceding section, each local spectral problem is assembled purely from the sparse symmetric subdomain matrix $A_i$ induced by the overlapping graph partition, together with partition-dependent diagonal modifications, which can be encoded by nodewise properties, $d_v$ and $s_v$.
This observation motivates a graph-based abstraction of the local spectral construction: we represent each subdomain by an attributed weighted graph whose connectivity and edge weights are given by the sparsity pattern and entries of $A_i$, while the node features are defined as a three-dimensional vector.
Specifically, for a node $v$ in the overlapping subdomain, this feature vector consists of the type of the node (set to 1 if $v$ lies on the subdomain boundary and 0 if in the interior), the value $d_v$, and the diagonal correction $s_v$ (defined as 0 if $v$ is in the interior).
On this representation, we employ a graph neural network to iteratively update node embeddings and to predict the multiscale basis functions associated with each subdomain, thereby providing a data-driven surrogate for the repeated solution of local generalized spectral problems within the two-level Schwarz framework.
In what follows, we denote the matrix encoding the connectivity and edge weights of the graph derived from an overlapping subdomain as $A$ (corresponding to $A_i$ defined above), and the node feature matrix as $Z$.

\subsection{Neural network architectures}
SP-LPMA GUNet \cref{fig:Global_Arch} (spectral-prior low-pass multiscale-attentive Graph U-Net) is built by integrating a spectral low-pass frontend, a Graph U-Net, and multi-head attention. 
Given the graph structure of each subdomain, the model directly learns local generalized eigenvectors that span the required coarse space.
Detailed architectural specifications are provided below.

\begin{figure}[!ht]
  \centering
  \includegraphics[width=\textwidth]{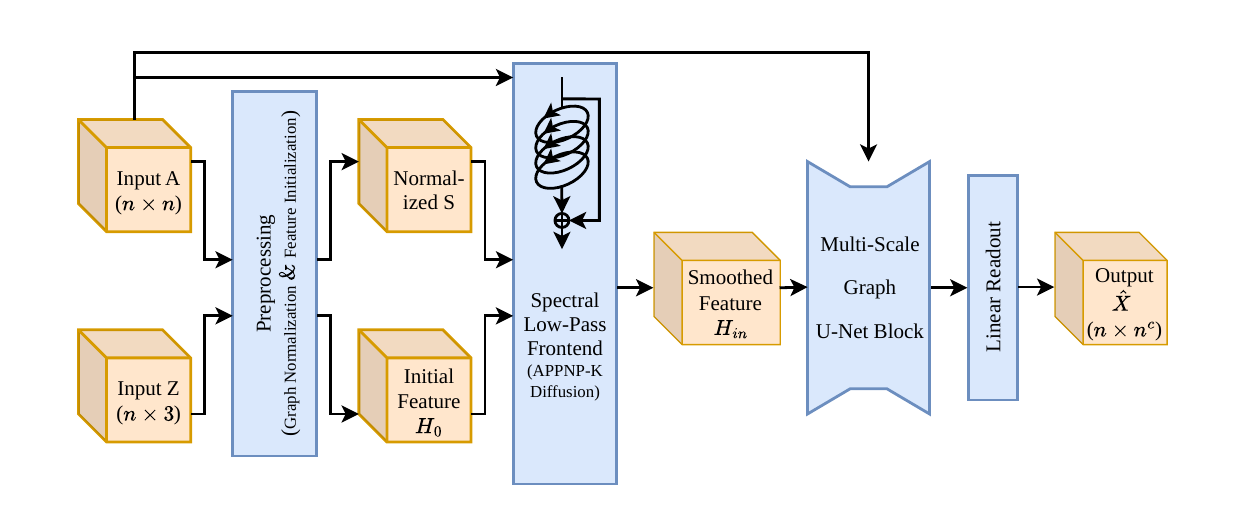}
  \caption{SP-LPMA GUNet Structure}
  \label{fig:Global_Arch}
\end{figure}

\begin{itemize}
    \item \textbf{Input:} $A \in \mathbb{R} ^ {n \times n}$, $Z \in \mathbb{R} ^ {n \times 3}$, where $n$ is the number of nodes in a subdomain.
    
    \item \textbf{Pre-processing:} 
    \begin{itemize}
        \item \textit{Graph Normalization:} We compute the degree matrix \(D = \Diag({d_v}) \in \mathbb{R}^{n \times n}\) and construct the normalized adjacency matrix:
        \[
            S = D^{-\frac{1}{2}} A D^{-\frac{1}{2}} \in \mathbb{R}^{n \times n}.
        \]
        \item \textit{Feature Initialization:} The input coordinates $Z$ are mapped to an initial high-dimensional feature space via a linear transformation:
        \[
            H_0 = Z W_{\mathrm{pre}} + b_{\mathrm{pre}} \in \mathbb{R}^{n \times d_0},
        \]
        where \(W_{\mathrm{pre}} \in \mathbb{R}^{3 \times d_0}\) and \(b_{\mathrm{pre}} \in \mathbb{R}^{d_0}\) are trainable weights and biases, and \(d_0\) is the initial hidden dimension.
    \end{itemize}
    
    \item \textbf{Spectral Low-Pass Frontend Block:}
    This block projects the columns of \(H_0\) (graph features) onto the low-frequency subspace spanned by the columns of \(S\) as follows \cite{gasteiger2018predict}:
    \begin{itemize}
        \item \textit{APPNP Diffusion:} We apply \(K\) steps of personalized PageRank-style diffusion. Let \(V_0 = H_0\), then for \(t=1, \dots, K\):
        \[
            V_t = S V_{t-1}, \quad H_{\mathrm{low}} = (1-\alpha) \sum_{t=0}^{K} \alpha^t V_t \in \mathbb{R}^{n \times d_0},
        \]
        where \(K \in \mathbb{N}^+\) (number of diffusion steps) and \(\alpha \in (0,1)\) (teleport probability) are predefined hyperparameters.
        \item \textit{Residual Normalization:} 
        \[
            H_{\mathrm{in}} = \mathrm{LayerNorm}\!\left(H_{\mathrm{low}} + \beta H_0\right) \in \mathbb{R}^{n \times d_0},
        \]
        where \(\beta\) is a hyperparameter. Here, \(\mathrm{LayerNorm}\) normalizes the features across the channel dimension for each node independently, which means that the mean and standard deviation are computed over each row vector of length \(d_0\).
    \end{itemize}

    \item \textbf{Multi-Scale Attention via Graph U-Net Block} (\cref{fig:GUnet_Arch}): 
    This module employs a U-shaped architecture of depth \(L\) for feature aggregation \cite{GUNET}.
    Let \(n_l\) and \(d_l\) denote the number of nodes and the feature dimension at resolution level \(l\) (with \(n_0=n\)).
    We employ \(N_{\mathrm{head}}\) attention heads in the signed attentive convolution \cite{huang2019signed}, where \(d_{\mathrm{head}}\) is the feature dimension of each head; hence the concatenated per-node output dimension at the level \(l\) is \(d_{l+1}=N_{\mathrm{head}}\times d_{\mathrm{head}}\).
    Let the initial state be \(H^{(0)} = H_{\mathrm{in}}\) and \(A^{(0)} = A\).
    
    \begin{itemize}
        \item \textbf{Encoder (Downscaling):} For each level \(l = 0, \dots, L-1\):
        \begin{itemize}
            \item \textit{Signed Attentive Convolution:} We decompose \(A^{(l)}\) into positive and negative components \(A^{(l)+} \) and \(A^{(l)-}\). Let \(\mathcal{N}^{\pm}(i)\) be the set of neighbors of node \(i\) in \(A^{(l)\pm}\). For each attention head \(h \in \{1, \dots, N_{\mathrm{head}}\}\):
            \begin{equation*}
            \begin{aligned}
                &Q_h = H^{(l)} W_h^{(l)} \in \mathbb{R}^{n_l \times d_{\mathrm{head}}}, \\
                &e_{ij}^{(h)\pm} = \tanh \big(Q_h[i,:] + Q_h[j,:]\big)  a_h + \gamma^{\pm} A^{(l)\pm}_{ij}, \\
                &\alpha_{ij}^{(h)\pm} = \exp(e_{ij}^{(h)\pm})/\sum_{k \in \mathcal{N}^{\pm}(i)} \exp(e_{ik}^{(h)\pm}), \\
                &O_h[i,:] = \sum_{j \in \mathcal{N}^{+}_{(i)}} \alpha_{ij}^{(h)+} Q_h[j,:] - \sum_{j \in \mathcal{N}^{-}_{(i)}} \alpha_{ij}^{(h)-} Q_h[j,:],
            \end{aligned}
            \end{equation*}
            where \(W_h^{(l)} \in \mathbb{R}^{d_l \times d_{\mathrm{head}}}\), \(a_h \in \mathbb{R}^{d_{\mathrm{head}}}\), and \(\gamma^{\pm} \in \mathbb{R}\) are trainable parameters. The outputs from all heads are concatenated along the feature dimension, denoted by \(\parallel\), followed by an Exponential Linear Unit (ELU) activation:
            \[
                H_{\mathrm{skip}}^{(l)} = \operatorname{ELU}\!\left( \big\|_{h=1}^{N_{\mathrm{head}}} O_h \right) \in \mathbb{R}^{n_l \times d_{l+1}}.
            \]
            
            \item \textit{Graph Pooling:} We compute score \(s^{(l)} = H_{\mathrm{skip}}^{(l)} w_{p}^{(l)} \in \mathbb{R}^{n_l}\) using trainable  vector \(w_{p}^{(l)} \in \mathbb{R}^{d_{l+1}}\). Define the index set \(\mathcal{I}^{(l+1)} \subset \{1, \dots, n_l\}\) corresponding to the nodes with the top \(k_l = \lfloor p_{\mathrm{rate}} n_l \rfloor\) highest scores, where \(p_{\mathrm{rate}} \in (0,1)\) is the pooling ratio hyperparameter. The coarsened graph and features are constructed by extracting the submatrix and sub-tensor indexed by \(\mathcal{I}^{(l+1)}\):
            \begin{equation*}
            \begin{aligned}
                & H^{(l+1)} = H_{\mathrm{skip}}^{(l)}[\mathcal{I}^{(l+1)}, :] \odot \operatorname{sigmoid}\!\big(s^{(l)}[\mathcal{I}^{(l+1)}]\big) \in \mathbb{R}^{n_{l+1} \times d_{l+1}}, \\
                & A^{(l+1)} = A^{(l)}[\mathcal{I}^{(l+1)}, \mathcal{I}^{(l+1)}] \in \mathbb{R}^{n_{l+1} \times n_{l+1}},
            \end{aligned}
            \end{equation*}
            where \(\odot\) denotes element-wise multiplication, and \(n_{l+1} = k_l\).
        \end{itemize}

        \item \textbf{Bottleneck Layer:} At the coarsest level \(L\), a residual fully-connected layer is applied:
        \begin{equation*}
            H_{\mathrm{dec}}^{(L)} = H^{(L)} + \left(H^{(L)} W_{\mathrm{fc}} + b_{\mathrm{fc}}\right) \in \mathbb{R}^{n_L \times d_L},
        \end{equation*}
        with trainable parameters \(W_{\mathrm{fc}} \in \mathbb{R}^{d_L \times d_L}\) and \(b_{\mathrm{fc}} \in \mathbb{R}^{d_L}\).

        \item \textbf{Decoder (Upscaling):} For each level \(l = L-1, \dots, 0\):
        \begin{itemize}
            \item \textit{Unpooling and Gated Skip Fusion:} We restore the spatial resolution by mapping the features \(H_{\mathrm{dec}}^{(l+1)}\) back to their original indices \(\mathcal{I}^{(l+1)}\) in an \(n_l\)-dimensional space, padding unselected nodes with zeros. Let this unpooled feature matrix be \(H_{\mathrm{up}}^{(l)} \in \mathbb{R}^{n_l \times d_{l+1}}\). It is then fused with the encoder's skip connection via a learned gating mechanism:
            \begin{equation*}
            \begin{aligned}
                & U^{(l)} = \left[ H_{\mathrm{up}}^{(l)} \parallel H_{\mathrm{skip}}^{(l)} \right] W_{s}^{(l)} \in \mathbb{R}^{n_l \times d_{l+1}}, \\
                & g^{(l)} = \operatorname{sigmoid}(U^{(l)} W_{g}^{(l)} + b_{g}^{(l)}) \in \mathbb{R}^{n_l \times d_{l+1}}, \\
                & Y^{(l)} = g^{(l)} \odot H_{\mathrm{up}}^{(l)} + (1 - g^{(l)}) \odot H_{\mathrm{skip}}^{(l)} \in \mathbb{R}^{n_l \times d_{l+1}},
            \end{aligned}
            \end{equation*}
            where \(W_{s}^{(l)} \in \mathbb{R}^{2d_{l+1} \times d_{l+1}}\) and \(W_{g}^{(l)} \in \mathbb{R}^{d_{l+1} \times d_{l+1}}\) are trainable weight matrices.
            \item \textit{Attentive Convolution:} The fused features are refined on the graph \(A^{(l)}\) using the same Signed Attentive Convolution operator defined in the encoder:
            \begin{equation*}
                H_{\mathrm{dec}}^{(l)} = \mathrm{AttnConv}\big(Y^{(l)}, A^{(l)}\big) \in \mathbb{R}^{n_l \times d_l}.
            \end{equation*}
        \end{itemize}
    \end{itemize}
    The final output of the U-Net block is \(H_{\mathrm{dec}}^{(0)} \in \mathbb{R}^{n \times d_0}\).
    
    \item \textbf{Linear Readout Block:} Map refined fine-scale features to the spectral subspace estimate of dimension $n_c$:
    \[
        \widehat{X} = H_{\mathrm{dec}}^{(0)} W_{\mathrm{out}} + b_{\mathrm{out}} \in \mathbb{R}^{n \times n_c},
    \]
    where \(W_{\mathrm{out}} \in \mathbb{R}^{d_0 \times n_c}\) and \(b_{\mathrm{out}} \in \mathbb{R}^{n_c}\) are trainable parameters.
\end{itemize}

\begin{figure}[!ht]
  \centering
  \includegraphics[width=0.8\textwidth]{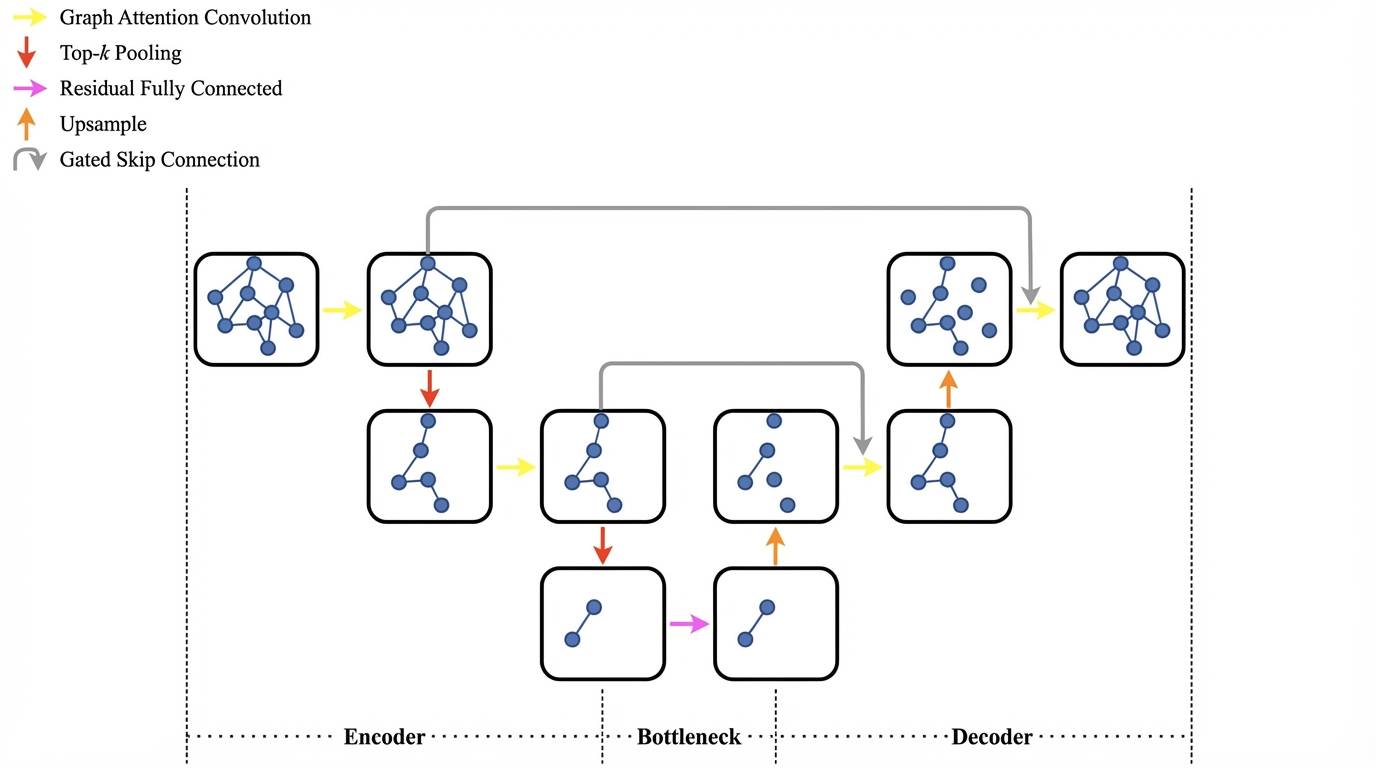}
\caption{Three Level Graph U-Net Structure}
\label{fig:GUnet_Arch}
\end{figure}

\subsection{Loss function}
The two-level Schwarz coarse space is spanned by locally constructed generalized multiscale basis functions extended to the global domain.
To enable direct learning of the basis function set within each local subdomain, we introduce a metric to quantify discrepancies between local coarse spaces.
Specifically, leveraging the inner product induced by the right-hand-side operator of the generalized local spectral problem, we define a projection-based distance between the target and predicted subspaces.
In neural network training, we adopt $\Dist_{\tilde{A}}\RoundBrackets{\text{target}, \text{prediction}}$  as the loss function; its precise definition is provided below.
\begin{definition}
  \label{defdist} Denote
  \[
    X_{j} = \Span \CurlyBrackets*{x_{j, 1},\ x_{j, 2}, \dots, x_{j,n^c}}, \quad
    Y_{j} = \SquareBrackets*{\psi_{j, 1}  \psi_{j, 2}  \dots \psi_{j, n^{c}}},
  \]
  where $X_{j}$ is the $j$-th local coarse space, $\CurlyBrackets{\psi_{j, 1},\ \psi_{j, 2}, \dots, \psi_{j, n^{c}}}$ is an orthonormal basis of subspace $X_{j}$ with respect to the inner product $\RoundBrackets{\cdot, \cdot}_{\tilde{A_j}}$ and they are arranged column-wise to form the matrix denoted as $Y_{j}$.
  For two local coarse spaces $X_j^{(1)}$ and $X_j^{(2)}$ defined on the same coarse element ($j$-th), we define the distance
  \begin{equation}
    \Dist_{\tilde{A}_j}\RoundBrackets{X_{j}^{(1)}, X_{j}^{(2)}}= \RoundBrackets*{n^{c} - \norm{\RoundBrackets{Y_{j}^{(1)}}^{\intercal} \tilde{A}_j Y_{j}^{(2)}}_F^{2}}^{1/2},
  \end{equation}
  where $\norm{\cdot }_{F}$ is Frobenius norm.
\end{definition}

For notational simplicity, we henceforth suppress the coarse-element index $j$ whenever no ambiguity can arise.

\begin{theorem}
The distance in \cref{defdist} enjoys the following properties:
\begin{itemize}
  \item Basis invariance: it is independent of the choice of $(\cdot,\cdot)_{\tilde{A}}$-orthonormal bases for the subspaces.
  \item Non-negativity: $n^{c} - \bigl\|(Y^{(1)})^{\intercal}\tilde{A}\,Y^{(2)}\bigr\|_{F}^{2} \ge 0$.
  \item Positive definiteness: $\Dist_{\tilde{A}}\RoundBrackets{X^{(1)},X^{(2)}}=0$ if and only if $X^{(1)}=X^{(2)}$ (as subspaces).
  \item Symmetry: $\Dist_{\tilde{A}}\RoundBrackets{X^{(1)},X^{(2)}}=\Dist_{\tilde{A}}\RoundBrackets{X^{(2)},X^{(1)}}$.
  \item Triangle inequality: for any $X^{(1)},X^{(2)},X^{(3)}$,
  \[
    \Dist_{\tilde{A}}\RoundBrackets{X^{(1)},X^{(2)}}
    + \Dist_{\tilde{A}}\RoundBrackets{X^{(2)},X^{(3)}}
    \;\ge\;
    \Dist_{\tilde{A}}\RoundBrackets{X^{(1)},X^{(3)}}.
  \]
\end{itemize}
\end{theorem}
The proof can be found in \cite{liu2024learning}.

Before demonstrating the suitability of the proposed loss function for the two-level Schwarz preconditioner, we first establish the necessary theoretical groundwork. 

To quantify the performance of the two-level Schwarz preconditioner, we present the following convergence analysis. The proof can be found in \cite{ADD,ADDThm}.

\begin{lemma}\label{cond}
    Following the notation in \cref{pre}, for any $i \in [1,k]$, assume $I_i \subseteq \Im(\tilde A_i)$ and $K_i \subseteq \ker(\tilde A_i)$, with $K_i$ containing $(\ker(D_i A_i D_i)\cap \ker(\tilde A_i))^{\perp} \cap \ker(\tilde A_i)$. 
    Let $\tau_i$ be a strictly positive real number and $\Pi_i$ be an orthogonal projection onto the subspace $S_i = K_i \oplus I_i$. 
    If the local coarse spaces in the two-level Schwarz preconditioner chosen as $S_i$ and the following inequality holds:
\begin{equation}
    (u - \Pi_i u)^\top D_i A_i D_i (u - \Pi_i u) \leq \tau_i u^\top \tilde A_i u \quad \forall u \in \mathbb{R}^{N_i},
\end{equation}
    the condition number of matrix preconditioned by \cref{ams} can be bounded by
\begin{equation}\label{bound}
    \kappa_2(M_{AMS}^{-1} A) \leq (k_c + 1)\left(2 + (2k_c + 1) k_m \max_i \tau_i\right),
\end{equation}
where $k_c$ denotes the minimum number of colors required to color the graph of A such that every two neighboring subdomains have different colors and $k_m$ denotes the maximum number of overlapping subdomains sharing a row of $A$.
\end{lemma}

Next, we show that the coefficient $\tau_i$ in \cref{bound} only depend on the subspace of the $\Im(\tilde A_i)$.

\begin{lemma}\label{conddep}
    Let \( B, \tilde{B} \in \mathbb{R}^{m \times m} \) be two symmetric positive semidefinite matrices.
    Let \( \ker(\tilde{B}) \), \( \Im(\tilde{B}) \) be the null space and image space of \( \tilde{B} \), respectively.
    Suppose subspaces $K$ and $I$ satisfy
\[
I \subseteq \Im(\tilde B), \quad (\ker(B)\cap \ker(\tilde B))^{\perp} \cap \ker(\tilde B) \subseteq K \subseteq \ker(\tilde B).
\]
Let $\Pi$ denote the orthogonal projector onto the subspace $S := K \oplus I$ and $\tau>0$ be arbitrary. Then the following two statements are equivalent.
\begin{itemize}
    \item For any $u \in \mathbb{R}^{m}$, we have
    \[
    (u - \Pi u)^\top B (u - \Pi u) \leq \tau u^\top \tilde B u.
    \]
    \item For any $v \in \Im(\tilde B)$, $\Pi_I$ is the $\tilde{B}$-orthogonal projector onto the subspace $I$, we have
    \[
    (v - \Pi_I v)^\top B (v - \Pi_I v) \leq \tau v^\top \tilde B v.
    \]
\end{itemize}
\end{lemma}
\begin{proof}
    Because \(\tilde{B}\) is symmetric positive semidefinite matrix, we have the orthogonal decomposition
    \[\Im(\tilde B) \oplus \ker(\tilde B) =   \mathbb{R}^{m}.\]
    \begin{itemize}
        \item [\((\Rightarrow)\)] Assume the first statement holds.
        Let \( v \in \Im(\tilde B) \subset \mathbb{R}^{m} \). Then \( v \) satisfies
        \[
        (v - \Pi v)^\top B (v - \Pi v) \leq \tau v^\top \tilde B v.
        \]
       By the definition of \( \Pi \) and \( S = K \oplus I \), the projection \( \Pi v \) admits a unique decomposition
       \[
       \Pi v = u_K + u_I, \quad \text{with } u_K \in K,\ u_I \in I.
       \]
       Since \( K \subseteq \ker(\tilde B) \), we have \( u_K \in \ker(\tilde B) \). But \( \Im(\tilde B) \perp \ker(\tilde B) \) and $v \in \Im(\tilde B)$, so \( u_K = 0 \). Hence, \( \Pi v = u_I = \Pi_I v\). We obtain
       \[
       (v - \Pi_I v)^\top B (v - \Pi_I v) \leq \tau v^\top \tilde B v.
       \]  
        \item[\((\Leftarrow)\)] Assume the second statement holds. Let \( u \in \mathbb{R}^{m} \). Write \( u = v + w \), where \( v \in \Im(\tilde B) \), \( w \in \ker(\tilde B) \). By assumption,
        \[
        (v - \Pi_I v)^\top B (v - \Pi_I v) \leq \tau v^\top \tilde B v.
        \]
        Due to \((\ker(B)\cap \ker(\tilde B))^{\perp} \cap \ker(\tilde B) = \ker^{\perp}(B)\cap\ker(\tilde B)\), we have \(\ker^{\perp}(B)\cap\ker(\tilde B) \subset K \subset \ker(\tilde B)\). By the definition of \( \Pi \) and $w \in \ker(\tilde B)$, we can get \(w - \Pi w \in \ker(\tilde B) \setminus K \subset \ker(\tilde B) \setminus (\ker^{\perp}(B)\cap\ker(\tilde B)) = \ker(\tilde B) \cap \ker(B)\). hence
        \[
        (u - \Pi u)^\top B (u - \Pi u) = (v - \Pi v)^\top B (v - \Pi v).
        \]
        Due to \(\Pi v = \Pi_I v\), 
        \[
        (u - \Pi u)^\top B (u - \Pi u) = (v - \Pi_I v)^\top B (v - \Pi_I v) \leq \tau u^\top \tilde B u.
        \]
    \end{itemize}
\end{proof}

\begin{theorem}
    Let $M_{AMS}^{-1}$ is the two-level Schwarz preconditioner with local generalized multiscale coarse spaces $X_i$, $\hat{X}_i$ and local coarse space $\hat{S}_i$ defined by
    \begin{itemize}
        \item $X_{i} = \Span \CurlyBrackets*{x_{i, 1}, x_{i, 2}, \dots, x_{i,n^c}}$, which is the subspace spanned by the eigenvectors, corresponding to the $n^c$ largest eigenvalues solved by \cref{eq:a_i_func}.
        \item $\hat{X}_{i} = \Span \CurlyBrackets*{\hat{x}_{i, 1},\ \hat{x}_{i, 2}, \dots, \hat{x}_{i,n^c}}$.
        \item $\hat{S}_{i} = ((\ker(D_i A_{i} D_i) \cap \ker(\tilde{A}_{i}))^{\perp} \cap \ker(\tilde{A}_{i}))\oplus \hat{X}_i$, where $A_i,D_i,\tilde{A_i}$ are the local operator for subdomian $\Omega_i^\delta$.
    \end{itemize}
    Then the condition number of matrix preconditioned by \cref{ams} can be bounded by
\begin{equation}\label{bound}
    \kappa_2(M_{AMS}^{-1} A) \leq (k_c + 1)\left(2 + (2k_c + 1) k_m \max_i (2\tau_i + 4 M_i\Dist^2_{\tilde{A}_i}(\hat{X_i}, X_i))\right)
\end{equation}
where $\tau_i$ is the $n^c$-th largest eigenvalue of the relevant local spectral problem, $M_i$ is the strictly positive real number such that $\norm{\cdot}^{2}_{D_iA_iD_i}\leq M_i\norm{\cdot}^{2}_{\tilde{A}_i}$ as introduced in \cite{ADDThm}, $k_c$ and $k_m$ are defined in \cref{cond}.
\end{theorem}
\begin{proof}
    For clarity, we omit the subdomain index $i$. 
    
    \paragraph{Notation and Projections}
    Let $T, \hat{T} \in \mathbb{R}^{n \times n^c}$ be matrices whose columns form $\tilde{A}$-orthonormal bases for the coarse spaces $X$ and $\hat{X}$, respectively. The $\tilde{A}$-orthogonal projections onto these spaces are given by $\Pi = T T^\top \tilde{A}$ and $\hat{\Pi} = \hat{T} \hat{T}^\top \tilde{A}$.
    We also define the associated Euclidean orthogonal projections in the transformed space:
    \[
        \Pi_0 = \tilde{A}^{1/2} T (\tilde{A}^{1/2} T)^\top, \quad \hat{\Pi}_0 = \tilde{A}^{1/2} \hat{T} (\tilde{A}^{1/2} \hat{T})^\top.
    \]
    Note that $\Pi_0$ and $\hat{\Pi}_0$ are symmetric idempotent matrices with rank $n^c$, satisfying $\tilde{A}^{1/2} \Pi = \Pi_0 \tilde{A}^{1/2}$ and $\tilde{A}^{1/2} \hat{\Pi} = \hat{\Pi}_0 \tilde{A}^{1/2}$.

    \paragraph{Error Decomposition.}
    We aim to bound the energy norm of the error $v - \hat{\Pi}v$. 
    By the triangle inequality, we have:
    \begin{equation}\label{eq:decomp}
        \|v - \hat{\Pi}v\|_{D A D}^2 \leq 2 \|v - \Pi v\|_{D A D}^2 + 2 \|\Pi v - \hat{\Pi}v\|_{D A D}^2.
    \end{equation}
    The first term represents the approximation error of the exact coarse space. By Assumption \cref{cond}, it is bounded by:
    \begin{equation}\label{eq:term1}
        \|v - \Pi v\|_{D A D}^2 \leq \tau \|v\|_{\tilde{A}}^2.
    \end{equation}

    \paragraph{Bound on Perturbation Error}
    For the second term in \cref{eq:decomp}, we utilize the norm equivalence $\| \cdot \|_{D A D}^2 \leq M \| \cdot \|_{\tilde{A}}^2$. Let $v_0 = \tilde{A}^{1/2} v$, then $\|v\|_{\tilde{A}} = \|v_0\|_2$. We derive:
    \begin{equation}\label{eq:term2_pre}
    \begin{aligned}
        \|\Pi v - \hat{\Pi}v\|_{D A D}^2 
        &\leq M \|\Pi v - \hat{\Pi}v\|_{\tilde{A}}^2 \\
        &= M \| \tilde{A}^{1/2} (\Pi - \hat{\Pi}) v \|_2^2 \\
        &= M \| (\Pi_0 - \hat{\Pi}_0) v_0 \|_2^2 \\
        &\leq M \|\Pi_0 - \hat{\Pi}_0\|_2^2 \, \|v\|_{\tilde{A}}^2.
    \end{aligned}
    \end{equation}
    
    \paragraph{Relation to Subspace Distance}
    It remains to bound the spectral norm $\|\Pi_0 - \hat{\Pi}_0\|_2$. We consider the Frobenius norm squared, utilizing the property that for any projection matrix $P$, $\operatorname{tr}(P^2) = \operatorname{tr}(P) = \operatorname{rank}(P)$:
    \[
    \begin{aligned}
        \|\Pi_0 - \hat{\Pi}_0\|_F^2 
        &= \operatorname{tr}(\Pi_0^2) + \operatorname{tr}(\hat{\Pi}_0^2) - 2\operatorname{tr}(\Pi_0 \hat{\Pi}_0) \\
        &= 2 n^c - 2\operatorname{tr}\left( (\tilde{A}^{1/2} T) (\tilde{A}^{1/2} T)^\top (\tilde{A}^{1/2} \hat{T}) (\tilde{A}^{1/2} \hat{T})^\top \right).
    \end{aligned}
    \]
    Using the cyclic property of the trace, $\operatorname{tr}(AB) = \operatorname{tr}(BA)$, the last term simplifies to:
    \[
        \operatorname{tr}\left( T^\top \tilde{A} \hat{T} \hat{T}^\top \tilde{A} T \right) = \| \hat{T}^\top \tilde{A} T \|_F^2.
    \]
    Thus, by the definition of the distance between subspaces in the $\tilde{A}$-inner product:
    \begin{equation}\label{eq:dist_rel}
        \|\Pi_0 - \hat{\Pi}_0\|_2^2 \leq \|\Pi_0 - \hat{\Pi}_0\|_F^2 = 2 \left( n^c - \| \hat{T}^\top \tilde{A} T \|_F^2 \right) = 2 \Dist_{\tilde{A}}^2(X, \hat{X}).
    \end{equation}
    
    \paragraph{Conclusion.}
    Substituting \cref{eq:term1}, \cref{eq:term2_pre}, and \cref{eq:dist_rel} back into \cref{eq:decomp}, we obtain:
    \[
        \|v - \hat{\Pi}v\|_{D A D}^2 \leq \left( 2\tau + 4 M \Dist_{\tilde{A}}^2(X, \hat{X}) \right) \|v\|_{\tilde{A}}^2.
    \]
    This completes the proof.
\end{proof}

\section{Numerical experiments}\label{nume}

All neural network training is conducted on a single NVIDIA RTX 5880 Ada Generation GPU, while the remaining iterative solving processes are performed on an Intel Core i9-12900 CPU.
This configuration is primarily adopted due to the relatively small scale of the cases, which limits the full utilization of the GPU's high parallel computing power. 
Additionally, this approach helps reduce unnecessary energy consumption and improves resource efficiency.
The neural network training is implemented in PyTorch, whereas the iterative solver relies on SciPy's sparse matrix routines.

\subsection{Multi-scale basis approximation experiments}
To facilitate efficient training without incurring the prohibitive costs of repeated PDE discretizations, we employ a synthetic algebraic data generation strategy designed to emulate the spectral and structural properties of discretized elliptic operators. 
We construct a dataset of random sparse weighted graphs, where the sparsity—defined by the number of nonzeros (NNZ) relative to the number of vertices $N$—is constrained to the interval $[3N, 7N]$. 
This range is rigorously selected to replicate the connectivity patterns characteristic of standard five-point (2D) and seven-point (3D) finite difference stencils.
Consistent with the domain decomposition strategy utilized in our preconditioner, we apply an overlapping graph partitioning scheme to these global graphs. 
This process yields a collection of weighted subgraphs $G_i=(V_i, E_i)$, with the partitioner tuned to enforce an average subdomain size of approximately \(2000\) vertices, consistent with the typical local problem dimension in our target large-scale applications. 
In total, we generate \(10000\) subdomain graphs, of which \(8000\) are used for training and the remaining \(2000\) are reserved for validation.

We systematically evaluate the SP-LPMA GUNet architecture (detailed in \cref{fig:Global_Arch}) to determine the optimal network depth. Specifically, we vary the internal depth of the Graph U-Net blocks (\cref{fig:GUnet_Arch}) while maintaining invariant architectural components and training protocols to ensure a controlled ablation study. The training dynamics, including loss trajectories and learning rate schedules, are visualized in \cref{fig:two-level 200,fig:three-level 200,fig:four-level 200,fig:four-level 300}. The convergence dynamics vary significantly with the hierarchical depth of the architecture. Within a baseline of 200 epochs, the two- and three-level models demonstrate robust convergence, with the validation loss reaching a steady asymptote. In contrast, the four-level model exhibits a slower convergence rate, failing to fully converge at 200 epochs (\cref{fig:four-level 200}). Although extending the training to 300 epochs allows the four-level configuration to stabilize (\cref{fig:four-level 300}), the resulting reduction in approximation error is marginal compared to the three-level model. This indicates that the deeper architecture complicates the optimization landscape without yielding a commensurate gain in expressive power for the target physical system. Consequently, the three-level model provides an optimal trade-off between numerical accuracy and computational efficiency, and is thus adopted for all subsequent simulations.

\begin{figure}[!ht]
    \centering
    \resizebox{\textwidth}{!}{
    \input{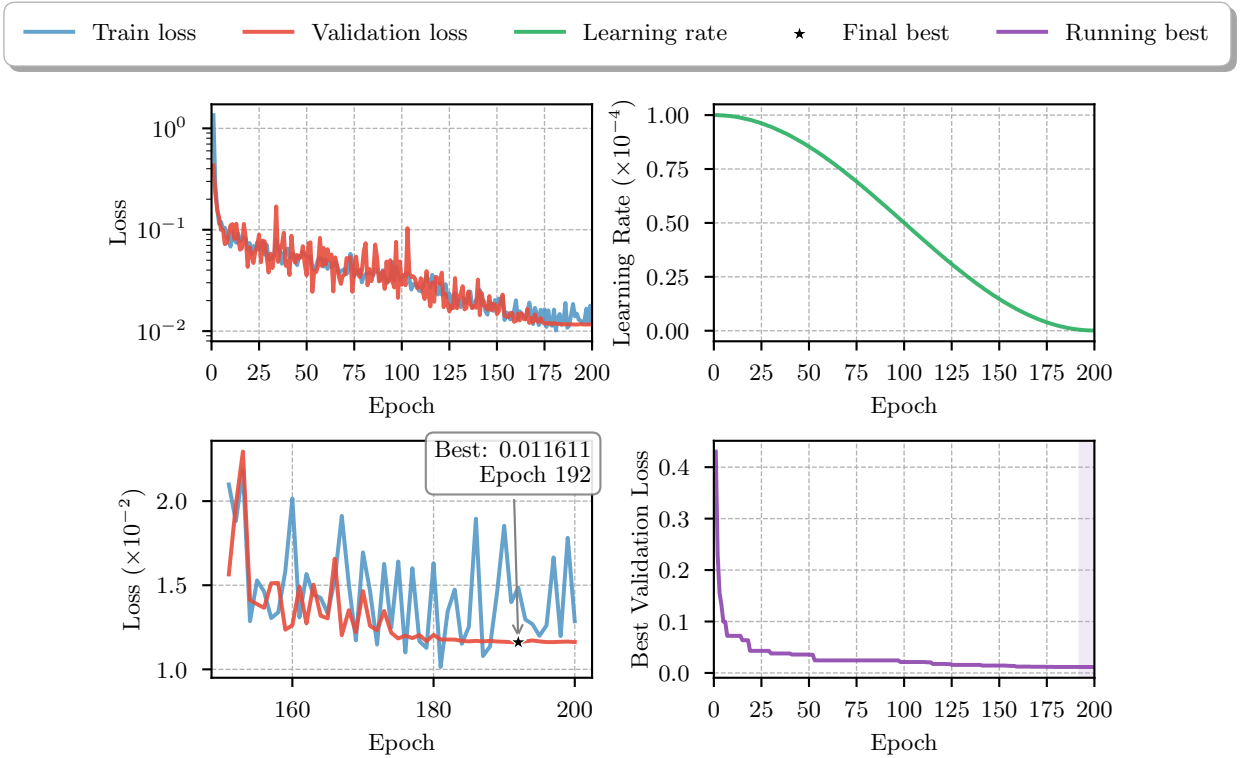}}
    \caption{Training results for the two level SP-LPMA GUNet over 200 epochs.}
    \label{fig:two-level 200}
\end{figure}

\begin{figure}[!ht]
    \centering
    \resizebox{\textwidth}{!}{
    \input{figures/training_dynamics_2.pgf}}
    \caption{Training results for the three level SP-LPMA GUNet over 200 epochs.}
    \label{fig:three-level 200}
\end{figure}

\begin{figure}[!ht]
    \centering
    \resizebox{\textwidth}{!}{
    \input{figures/training_dynamics_3_epoch200.pgf}}
    \caption{Training results for the four level SP-LPMA GUNet over 200 epochs.}
    \label{fig:four-level 200}
\end{figure}

\begin{figure}[!ht]
    \centering
    \resizebox{\textwidth}{!}{
    \input{figures/training_dynamics_3.pgf}}
    \caption{Training results for the four level SP-LPMA GUNet over 300 epochs.}
    \label{fig:four-level 300}
\end{figure}

To validate the approximation fidelity of the selected three-level model, we conduct a qualitative comparison between the predicted multiscale basis functions and the reference bases derived from exact local generalized eigenvalue problems. \Cref{fig:basis_comparison} illustrates this comparison on a representative graph subdomain. The visual agreement confirms that the learned model effectively captures the dominant spectral modes required for the construction of the coarse space.
\begin{figure}[!ht]
  \centering
  \includegraphics[width=\textwidth]{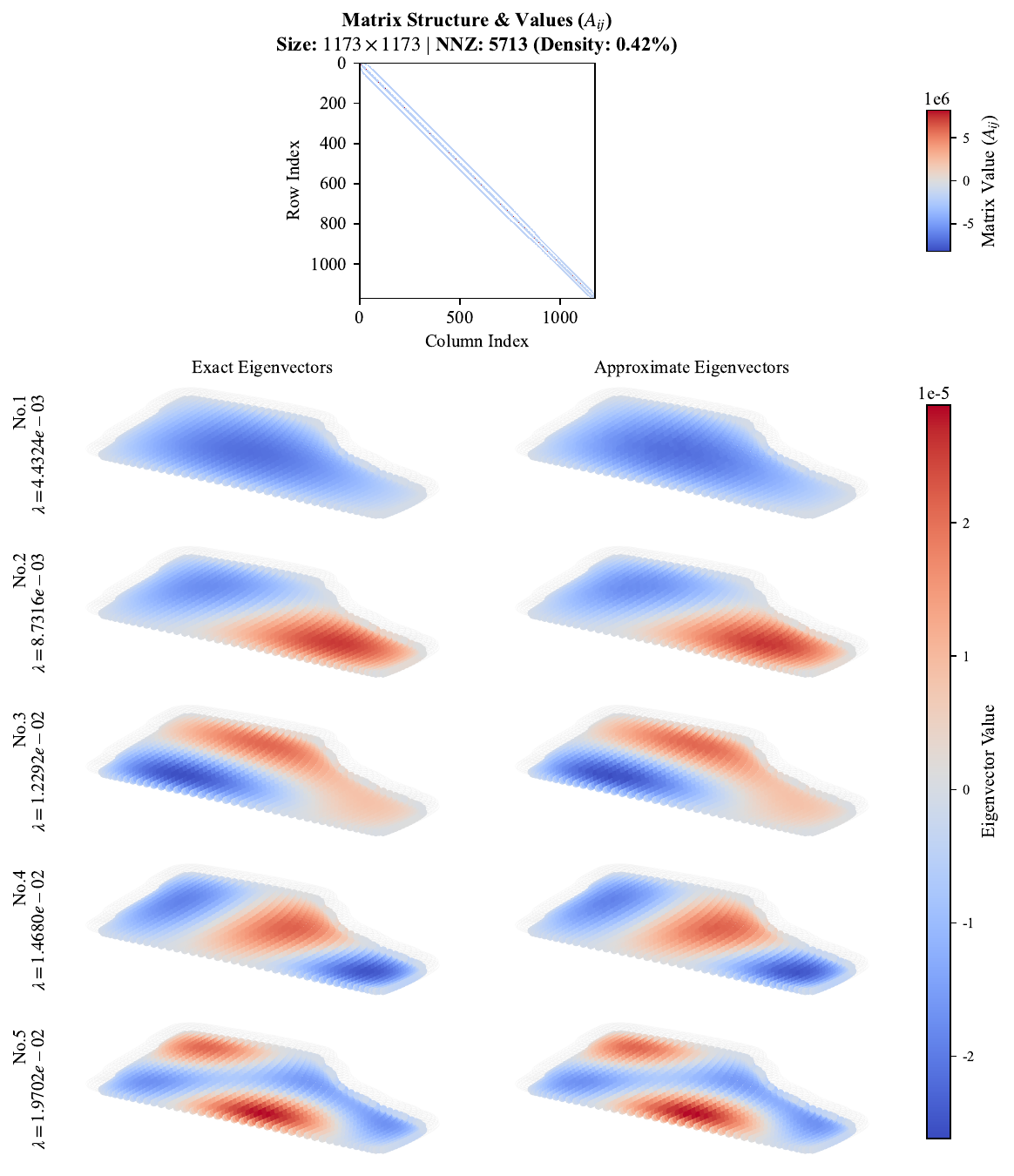}
  \caption{Qualitative evaluation of multiscale basis functions approximation on a representative graph subdomain.}
  \label{fig:basis_comparison}
\end{figure}

\subsection{Preconditioner comparison}
\label{subsec:precond_comparison}

We assess the proposed learning-based multiscale basis functions by analyzing the convergence of a two-level additive Schwarz preconditioner for the Darcy flow problem with highly heterogeneous permeability. Let $\Omega \subset \mathbb{R}^d$ ($d \in {2, 3}$) be a bounded Lipschitz domain with boundary $\partial\Omega = \Gamma_D \cup \Gamma_N$ ($\Gamma_D \cap \Gamma_N = \emptyset$). The mixed boundary value problem is given by \cref{eq:original_equation}.

Discretization is performed using a cell-centered finite volume method on structured Cartesian grids. 
A velocity-elimination procedure based on trapezoidal quadrature \cite{vet} reduces the system to a pressure-only Two-Point Flux Approximation (TPFA) scheme. 
Interface transmissibilities are computed via harmonic averaging to accommodate the heterogeneity of $\kappa(\bm{x})$. 
Dirichlet and Neumann conditions are enforced via numerical flux modifications and prescribed boundary fluxes, respectively. 
This formulation yields a sparse linear system $Ap=b$, where $A$ is symmetric positive definite (SPD) provided $\kappa$ is scalar and $\Gamma_D$ has non-zero measure.

\subsubsection{2D experiments}
\label{subsubsec:2d}

We consider steady Darcy flow on the unit square \(\Omega=(0,1)^2\), discretized on a uniform \(2048\times 2048\) grid. 
To assess robustness and out-of-distribution generalization with respect to boundary forcing and medium complexity, we test two orthogonal boundary configurations and two permeability families.

\paragraph{Boundary configurations}
We denote by \(\mathcal{C}_1\) and \(\mathcal{C}_2\) two Dirichlet--Neumann decompositions inducing flow along the coordinate axes:
\begin{itemize}
  \item \textbf{Configuration \(\mathcal{C}_1\) (flow in \(y\)-direction).}
  Dirichlet conditions are imposed on \(\{y=0\}\cup\{y=1\}\) with \(p_D=0\) at \(y=0\) and \(p_D=1\) at \(y=1\); homogeneous Neumann conditions (\(g=0\)) are applied on \(\{x=0\}\cup\{x=1\}\).
  \item \textbf{Configuration \(\mathcal{C}_2\) (flow in \(x\)-direction).}
  Dirichlet conditions are imposed on \(\{x=0\}\cup\{x=1\}\) with \(p_D=1\) at \(x=0\) and \(p_D=0\) at \(x=1\); homogeneous Neumann conditions (\(g=0\)) are applied on \(\{y=0\}\cup\{y=1\}\).
\end{itemize}

\paragraph{Permeability family I: log-normal random fields.}
We generate continuous heterogeneous media by setting \(\kappa(\mathbf{x})=\exp(Z(\mathbf{x}))\), where \(Z\) is a zero-mean Gaussian random field with anisotropic exponential covariance
\[
C_Z(\mathbf{x},\mathbf{y})
=\sigma^2\exp\!\left(
-\sqrt{\frac{(x_1-y_1)^2}{\eta_1^2}+\frac{(x_2-y_2)^2}{\eta_2^2}}
\right),
\]
with \(\sigma^2=2\) and correlation lengths \(\eta_1,\eta_2\).
\begin{figure}[!ht]
  \centering
  \includegraphics[width=\textwidth]{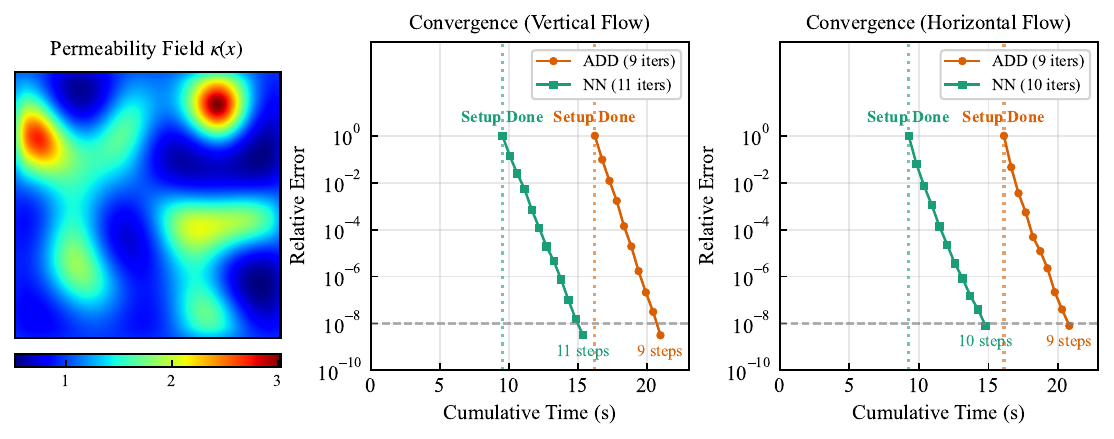}
  \caption{Performance on a 2D heterogeneous log-normal permeability field.
\textbf{Left:} permeability realization $\kappa(\mathbf{x})$ generated from a Gaussian random field $Z=\log\kappa$ with exponential covariance (Section~\ref{subsubsec:2d}) on a $2048\times 2048$ mesh.
\textbf{Middle/Right:} time-to-accuracy comparison between ADD and NN for $\mathcal{C}_1$ (vertical flow) and $\mathcal{C}_2$ (horizontal flow).
Curves report relative error versus cumulative time (setup + solve); dotted lines mark the end of setup and the dashed line denotes the target tolerance $10^{-8}$.}
  \label{fig:res_RF 2D}
\end{figure}
Figure~\ref{fig:res_RF 2D} reports time-to-accuracy curves under \(\mathcal{C}_1\) and \(\mathcal{C}_2\), together with a representative permeability realization.

\paragraph{Permeability family II: high-permeability channels (DFN-like).}
We consider discontinuous, high-contrast media with a background matrix permeability \(\kappa_m=1\) and \(N_c\) randomly placed channels, where \(N_c\sim \mathrm{Unif}\{8,\ldots,20\}\).
Each channel is a line segment with length \(L\sim \mathrm{Unif}[300,1000]\) (grid units) and width \(w\sim \mathrm{Unif}[3,8]\), assigned permeability \(\kappa_c\), while the remaining region keeps \(\kappa_m\).
We test two contrast ratios to quantify robustness as discontinuities and preferential paths become more severe.
\begin{figure}[!ht]
  \centering
  \includegraphics[width=\textwidth]{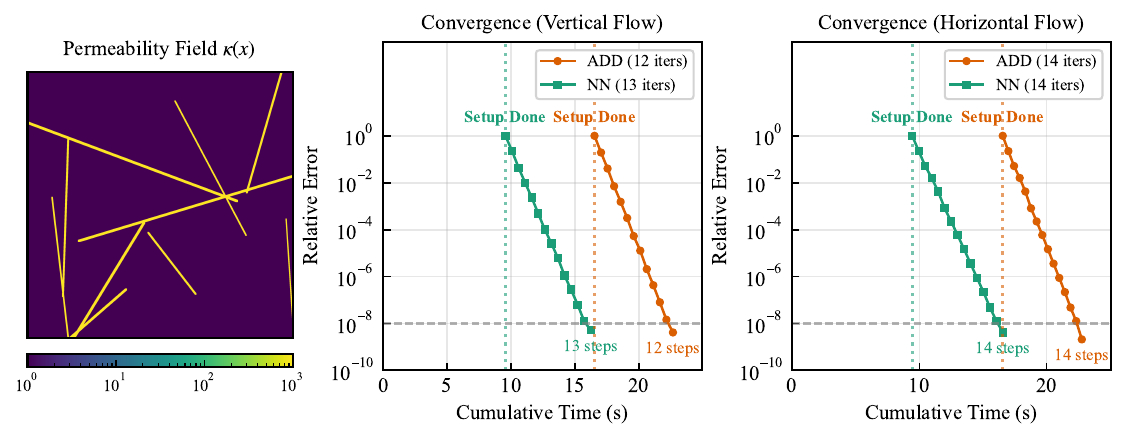}
  \caption{Performance on a 2D high-contrast discrete fracture network (DFN).
\textbf{Left:} Permeability realization $\kappa(\mathbf{x})$ featuring randomly distributed high-conductivity fractures (channels with $\kappa=10^3$) embedded in a homogeneous background matrix on a $2048\times 2048$ mesh.
\textbf{Middle/Right:} Time-to-accuracy comparison between ADD and NN for $\mathcal{C}_1$ (vertical flow) and $\mathcal{C}_2$ (horizontal flow).
Curves report relative error versus cumulative time (setup + solve); dotted lines mark the end of setup and the dashed line denotes the target tolerance $10^{-8}$.}
  \label{fig:res_RB 1e3 2D}
\end{figure}

\begin{figure}[!ht]
  \centering
  \includegraphics[width=\textwidth]{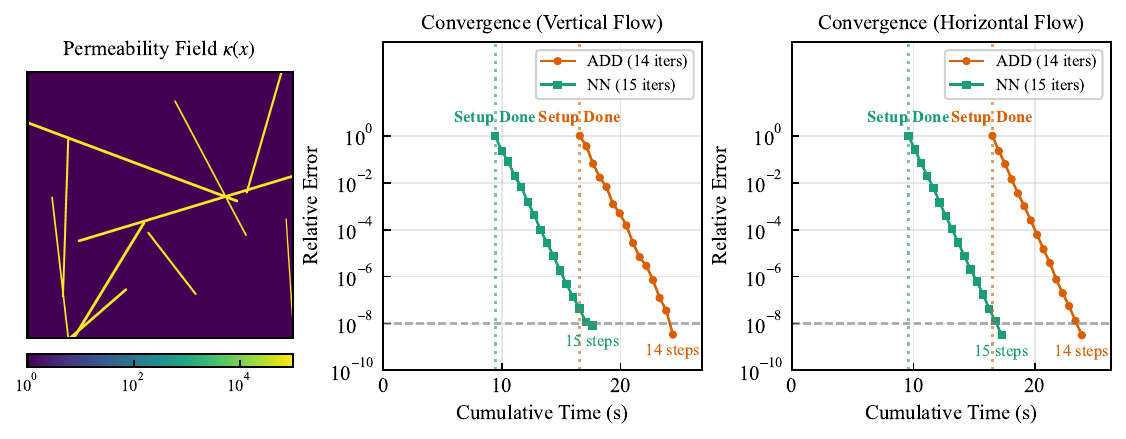}
  \caption{Performance on a 2D high-contrast discrete fracture network (DFN).
\textbf{Left:} Permeability realization $\kappa(\mathbf{x})$ featuring randomly distributed high-conductivity fractures (channels with $\kappa=10^5$) embedded in a homogeneous background matrix on a $2048\times 2048$ mesh.
\textbf{Middle/Right:} Time-to-accuracy comparison between ADD and NN for $\mathcal{C}_1$ (vertical flow) and $\mathcal{C}_2$ (horizontal flow).
Curves report relative error versus cumulative time (setup + solve); dotted lines mark the end of setup and the dashed line denotes the target tolerance $10^{-8}$.}
  \label{fig:res_RB 1e5 2D}
\end{figure}
Results are shown in Figure~\ref{fig:res_RB 1e3 2D} for \(\kappa_c=10^3\) and in Figure~\ref{fig:res_RB 1e5 2D} for \(\kappa_c=10^5\), each under \(\mathcal{C}_1\) and \(\mathcal{C}_2\).

\subsubsection{3D experiments}
\label{subsubsec:3d}

We next consider \(\Omega=(0,1)^3\) on a \(256\times256\times256\) grid, using the same evaluation protocol: two orthogonal boundary configurations and the same two permeability families extended to three dimensions.

\paragraph{Boundary configurations}
Let \(\partial\Omega=\Gamma_D^{(k)}\cup\Gamma_N^{(k)}\) denote the Dirichlet--Neumann decomposition for configuration \(\mathcal{C}_k\), \(k\in\{1,2\}\):
\begin{itemize}
  \item \textbf{Configuration \(\mathcal{C}_1\) (flow in \(y\)-direction).}
  \[
    \Gamma_D^{(1)}=[0,1]\times\{0,1\}\times[0,1],
    \qquad
    \Gamma_N^{(1)}=\partial\Omega\setminus \Gamma_D^{(1)}.
  \]
  We impose \(g=0\) on \(\Gamma_N^{(1)}\) and set \(p_D=0\) at \(y=0\), \(p_D=1\) at \(y=1\).
  \item \textbf{Configuration \(\mathcal{C}_2\) (flow in \(x\)-direction).}
  \[
    \Gamma_D^{(2)}=\{0,1\}\times[0,1]\times[0,1],
    \qquad
    \Gamma_N^{(2)}=\partial\Omega\setminus \Gamma_D^{(2)}.
  \]
  We impose \(g=0\) on \(\Gamma_N^{(2)}\) and set \(p_D=1\) at \(x=0\), \(p_D=0\) at \(x=1\).
\end{itemize}

\paragraph{Permeability family I: log-normal random fields}
We sample \(\kappa(\mathbf{x})=\exp(Z(\mathbf{x}))\) with zero-mean GRF \(Z\) and 3D anisotropic exponential covariance
\[
C_Z(\mathbf{x},\mathbf{y})
=\sigma^2\exp\!\left(
-\sqrt{
\frac{(x_1-y_1)^2}{\eta_1^2}+
\frac{(x_2-y_2)^2}{\eta_2^2}+
\frac{(x_3-y_3)^2}{\eta_3^2}}
\right),
\]
where \(\sigma^2=2\) and \(\eta_1,\eta_2,\eta_3\) are correlation lengths.
\begin{figure}[!ht]
  \centering
  \includegraphics[width=\textwidth]{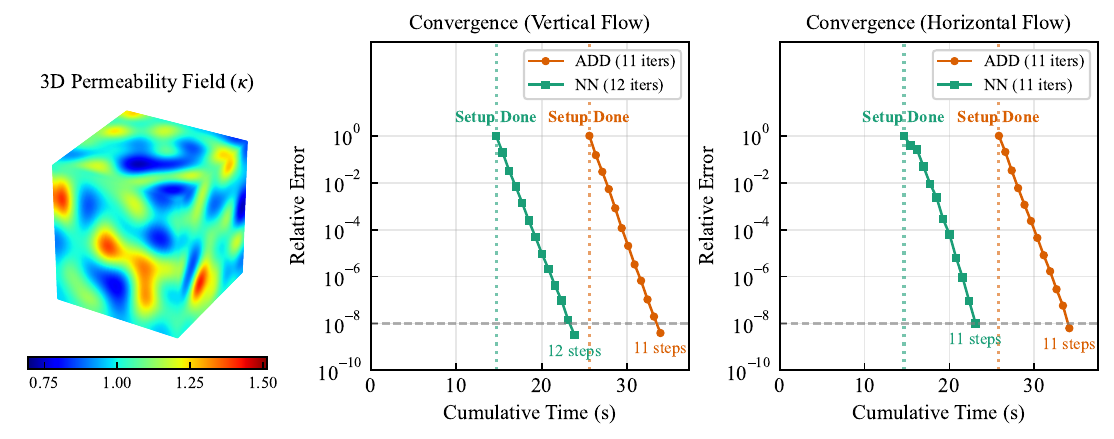}
  \caption{Performance on a 3D heterogeneous log-normal permeability field.
\textbf{Left:} permeability realization $\kappa(\mathbf{x})$ generated from a Gaussian random field $Z=\log\kappa$ with exponential covariance (Section~\ref{subsubsec:3d}) on a $256 \times 256 \times 256$ mesh.
\textbf{Middle/Right:} time-to-accuracy comparison between ADD and NN for $\mathcal{C}_1$ (vertical flow) and $\mathcal{C}_2$ (horizontal flow).
Curves report relative error versus cumulative time (setup + solve); dotted lines mark the end of setup and the dashed line denotes the target tolerance $10^{-8}$.}
  \label{fig:res_RF 3D}
\end{figure}
Figure~\ref{fig:res_RF 3D} summarizes the corresponding time-to-accuracy results under \(\mathcal{C}_1\) and \(\mathcal{C}_2\).

\paragraph{Permeability family II: high-permeability channels (3D)}
We construct discontinuous 3D channelized media with \(\kappa_m=1\) and randomly distributed tubular channels of permeability \(\kappa_c\).
We sample \(N_c\sim \mathrm{Unif}\{6,\ldots,15\}\); each channel starts from a random line segment of length \(L\sim \mathrm{Unif}[80,200]\), then is dilated to a tube with radius \(r\sim \mathrm{Unif}[2,5]\).
We again test \(\kappa_c\in\{10^3,\,10^5\}\).
\begin{figure}[!ht]
  \centering
  \includegraphics[width=\textwidth]{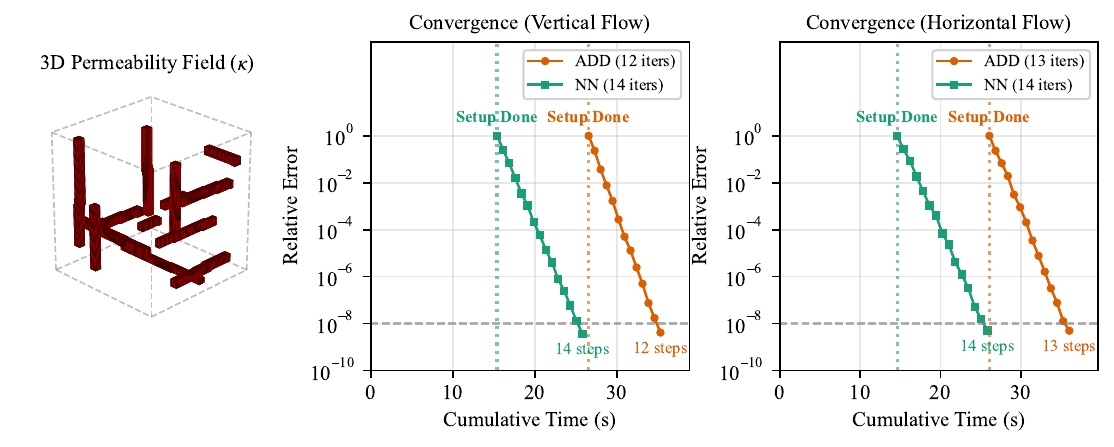}
  \caption{Performance on a 3D high-contrast discrete fracture network (DFN).
\textbf{Left:} Permeability realization $\kappa(\mathbf{x})$ featuring randomly distributed high-conductivity fractures (channels with $\kappa=10^3$) embedded in a homogeneous background matrix on a $256\times 256\times256$ mesh.
\textbf{Middle/Right:} Time-to-accuracy comparison between ADD and NN for $\mathcal{C}_1$ (vertical flow) and $\mathcal{C}_2$ (horizontal flow).
Curves report relative error versus cumulative time (setup + solve); dotted lines mark the end of setup and the dashed line denotes the target tolerance $10^{-8}$.}
  \label{fig:res_RB 1e3 3D}
\end{figure}

\begin{figure}[!ht]
  \centering
  \includegraphics[width=\textwidth]{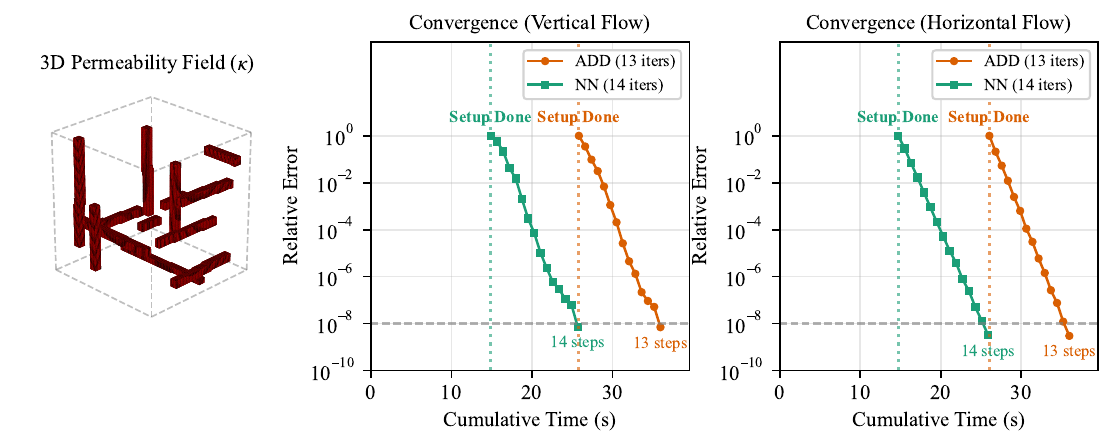}
  \caption{Performance on a 3D high-contrast discrete fracture network (DFN).
\textbf{Left:} Permeability realization $\kappa(\mathbf{x})$ featuring randomly distributed high-conductivity fractures (channels with $\kappa=10^5$) embedded in a homogeneous background matrix on a $256\times 256\times256$ mesh.
\textbf{Middle/Right:} Time-to-accuracy comparison between ADD and NN for $\mathcal{C}_1$ (vertical flow) and $\mathcal{C}_2$ (horizontal flow).
Curves report relative error versus cumulative time (setup + solve); dotted lines mark the end of setup and the dashed line denotes the target tolerance $10^{-8}$.}
  \label{fig:res_RB 1e5 3D}
\end{figure}
Time-to-accuracy curves are reported in Figure~\ref{fig:res_RB 1e3 3D} (\(\kappa_c=10^3\)) and Figure~\ref{fig:res_RB 1e5 3D} (\(\kappa_c=10^5\)), under both \(\mathcal{C}_1\) and \(\mathcal{C}_2\).

For each case, although the NN-accelerated basis represents an approximate multiscale subspace and can lead to a mildly larger iteration count, the preconditioner setup is significantly faster, resulting in consistently improved end-to-end time-to-accuracy.

\section{Conclusion}\label{con}

We developed a purely algebraic, learning-based coarse-space construction for two-level overlapping Schwarz preconditioning of high-contrast Darcy systems.
The key idea is to replace the repeated solution of local generalized eigenproblems during setup by a graph neural network that predicts local coarse spaces from the system-matrix graph.
On the theoretical side, we introduced a coefficient-weighted subspace-distance measure and established a condition-number bound for the resulting preconditioned operator in terms of this distance.
This result provides a principled objective for supervised training and connects the learning error to solver performance.
In our experiments, the training objective converged to a final loss of \(9.4\times 10^{-3}\).
Across the considered permeability contrasts and mixed boundary conditions, the proposed approach consistently reduces the dominant setup cost.
In both 2D and 3D, the setup time is typically reduced by about \(75\%\)–\(80\%\), which translates into an overall end-to-end time-to-solution reduction of about \(25\%\)–\(30\%\).
While the learned coarse space slightly increase the Krylov solve time in some cases, the total time is consistently improved because the setup phase is substantially accelerated.
These results indicate that learning-based coarse-space construction can deliver robust and practically efficient two-level Schwarz preconditioning for heterogeneous porous-media flow simulations.

\section*{Acknowledgement}

The research of Eric Chung is partially supported by the Hong Kong RGC General Research Fund (Projects: 14305423 and 14305624), as well as the 1+1+1 CUHK-CUHK(SZ)-GDSTC Joint Collaboration Fund (Project: 2025A0505000059).

\end{document}